\documentclass[11pt]{article}
\usepackage[a4paper, margin=1.1in]{geometry}
\usepackage{amsmath,amsthm}
\usepackage{amsfonts,amssymb}
\usepackage{graphicx}
\usepackage{csquotes}
\usepackage{breqn}
\usepackage{amsmath}
\usepackage{float}
\usepackage{caption}
\usepackage{chngcntr}

\usepackage{tikz-qtree}
\usetikzlibrary{arrows.meta}
\usepackage{upgreek}
\usepackage{epsfig,amsmath,amsthm,latexsym,amssymb,amsgen,graphicx}
\usepackage{verbatim}
\usepackage{cite}
\theoremstyle{plain}% Theorem-like structures provided by amsthm.sty
\newtheorem{theorem}{Theorem}[section]
\newtheorem{lemma}[theorem]{Lemma}

\theoremstyle{definition}
\newtheorem{definition}[theorem]{Definition}

\theoremstyle{remark}

\begin{document}

\title
%[Second Grade Fluid]
{\bf{ On the distance spectrum and distance-based topological indices of central vertex-edge join of three graphs}}
\date{}
\author{Haritha T$^1$\footnote{harithathottungal96@gmail.com},  Chithra A. V$^1$\footnote{chithra@nitc.ac.in}
	\\ \\ \small 
 1 Department of Mathematics, National Institute of Technology Calicut,\\\small Calicut-673 601, Kerala, India\\ \small}
%\keywords{}
\maketitle
\begin{abstract}
  Topological indices are molecular descriptors that describe the properties of chemical compounds. These topological indices correlate specific physico-chemical properties like boiling point, enthalpy of vaporization, strain energy, and stability of chemical compounds. This article introduces a new graph operation based on central graph called central vertex-edge join and provides its results related to graph invariants like eccentric-connectivity index, connective eccentricity index, total-eccentricity index, average eccentricity index, Zagreb eccentricity indices, eccentric geometric-arithmetic index, eccentric
atom-bond connectivity index, and Wiener index. Also, we discuss the distance spectrum of the central vertex-edge join of three regular graphs. Furthermore, we obtain new families of $D$-equienergetic graphs, which are non $D$-cospectral.
\end{abstract}
 \section{Introduction}
All graphs considered in this paper are simple connected and un-directed. Let $G_n= (V(G_n),E(G_n))$  be a graph of order $n$ and size $m$, with vertices $v_i$, $1\leq i\leq n$ and edges $e_j$, $1\leq j\leq m$. The \textit{adjacency matrix} of a graph $G_n$ is a square symmetric matrix $A(G_n)= (a_{ij})_{n\times n}$, where $a_{ij}=1$ when $v_i$ is adjacent to $v_j$ and $a_{ij}=0$ otherwise. Denote the eigenvalues of $A(G_n)$ by $\theta_1\geq \theta_2\geq\cdots\geq \theta_n$, then the collection of all the eigenvalues of $A(G_n)$ is said to be the adjacency spectrum of $G_n$. The \textit{incidence matrix} of a graph $G_n$ is the $n \times m$ matrix $Q(G_n)= (q_{ij})_{n\times m}$, where $q_{ij}= 1$ when $v_i$ is incident to $e_j$ and $q_{ij}= 0$ otherwise. The \textit{line graph} $L(G_n)$ of $G_n$ is a graph with $V(L(G_n))= E(G_n)$, in which two vertices in $L(G_n)$ are adjacent only if the corresponding edges in $G_n$ have an end vertex in common. For a partitioned matrix we denote its \textit{equitable quotient matrix} by $M$ \textnormal{\cite{brouwer2011spectra}} and all one entry matrix by $J_{n\times m}$. We represent by $C_n$ the \textit{cycle}, $K_{p,q}$ $\left(p+q= n\right)$ the \textit{complete bipartite graph}, and $K_n$ the \textit{complete graph} each of order $n.$ \\
 \par The \textit{distance} $d_G(v_i,v_j)$ between two vertices is the length of the shortest path connecting them in $G_n$ \textnormal{\cite{buckley1990distance}}. In particular, for any $v_i\in V(G_n), \;d_{G_n}(v_i, v_i)=0$. The eccentricity of a vertex $v_i\in V(G_n)$ is defined as $ec_{G_n}(v_i)= max\{d(v_i,v_j): v_j\in V(G_n)\}.$ The square symmetric matrix $D(G_n)= (d_{ij})_{n\times n}= d_{G_n}(v_i, v_j)$ is called the \textit{distance matrix} of $G_n$ and its eigenvalues are said to be the distance eigenvalues ($D$-eigenvalues). Let $\mu_1\geq \mu_2\geq\cdots \geq \mu_t$ be the distinct $D$-eigenvalues of $G_n$ with multiplicities $s_1, s_2,\ldots, s_t$. Then the \textit{distance spectrum}($D$-spectrum) of $D(G_n)$ is given by $Spec_{D}(G_n)= \begin{pmatrix}\mu_1&  \mu_2&\dots&\mu_t\\ s_1& s_2&\dots& s_t\end{pmatrix}$. The non-isomorphic graphs of the same order having identical $D$-spectrum are called \textit{$D$-cospectral graphs}. The \textit{transmission} $Tr(v_i)$ of a vertex $v_i$  is the sum of all distances from $v_i$ to all other vertices in $G.$ See \textnormal{\cite{aouchiche2014distance}} for a survey of the distance spectra of graphs.\\
 \par In graph theory, operations on graphs play a significant role. For significant works on distance spectra of graph operations see \textnormal{\cite{aalipour2016distance,central,indulala2019distance}}. The \textit{central graph} $C(G_n)$ of $G_n$ \textnormal{\cite{vivin2008harmonious}}, is the graph obtained by adding new vertex to each edge of $G_n$ and joining all the non adjacent vertices in $G_n$. In \textnormal{\cite{jahfar2020central}}, the authors defined the central vertex(edge) join of two graphs and their adjacency spectra are investigated when $G_{n_1}$ and $G_{n_2}$ are both regular graphs. Inspired by these works, we introduce a new graph operation based on central graph and join.
 \begin{definition}
 Let $G_{n_1}$, $G_{n_2}$ and $G_{n_3}$ be any three graphs of orders $n_1$, $n_2$, $n_3$ and sizes $m_1$, $m_2$, $m_3$ respectively. The \textit{central vertex-edge join} (CVE-join) of $G_{n_1}$ with $G_{n_2}$ and $G_{n_3}$, is the graph $G_{n_1}^C \triangleright (G_{n_2}^V\cup G_{n_3}^E)$, is obtained from $C(G_{n_1})$, $G_{n_2}$ and $G_{n_3}$ by joining each vertex of $G_{n_1}$ with every vertex of $G_{n_2}$ and each vertex corresponding to the edges of $G_{n_1}$ with every vertex of $G_{n_3}.$ The set of vertices in $C(G_{n_1})$ corresponding to the edges of $G_{n_1}$ is denoted by $I(G_{n_1}).$
 \end{definition}
The order and size of $G_{n_1}^C \triangleright (G_{n_2}^V\cup G_{n_3}^E)$ are $n_1+m_1+n_2+n_3$ and $m_1+m_2+m_3+n_1n_2+m_1n_3+\frac{n_1(n_1-1)}{2}$ respectively. 
 \begin{center}

 \begin{tikzpicture}[scale=0.4,inner sep=2pt]
\draw (0,0) node(11) [circle,draw,fill] {}
      (0,4) node(3) [circle,draw,fill] {}
      (0,6) node(5) [circle,draw] {}
      (0,8) node(1) [circle,draw,fill] {}
      (0,12) node(9) [circle,draw,fill] {}
      (4,4) node(6) [circle,draw] {}
      (8,0) node(12) [circle,draw,fill] {}
      (8,4) node(4) [circle,draw,fill] {}
      (8,6) node(7) [circle,draw] {}
      (8,8) node(2) [circle,draw,fill] {}
      (8,12) node(10) [circle,draw,fill] {}
      (4,8) node(8) [circle,draw] {};
            
\draw [-] (10) to (1) to (5) to (3) to (6) to (4) to (7) to (2) to (8) to (1);  
\draw [-] (1) to (9) to (10) to (2) to (9);
\draw [-] (1) to (4);
\draw [-] (3) to (2);
\draw [-] (5) to (12);
\draw [-,bend left] (3) to (9);
\draw[-,bend right] (9) to (4);
\draw[-,bend left] (10) to (3);
\draw[-, bend right] (4) to (10);
\draw[-] (7) to (11) to (12) to (6) to (11) to (8) to (12);
\draw[-, bend left] (11) to (5);

\draw[-, bend right] (12) to (7);

\end{tikzpicture}
    
 \end{center}
\begin{center}
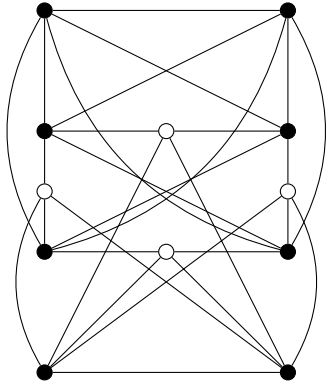
Figure $1.1$. $C_{4}^C\triangleright ( K_2^V \cup K_2^E).$
\end{center}
In general, finding the $D$-spectrum of complicated networks is a difficult problem. In Section $2$, we obtain the $D$-spectrum of CVE-join of three regular graphs. Using the results mentioned above, we can find the $D$-spectra of more complicated networks.
 
 \par The \textit{distance energy} ($D$-energy) \textnormal{\cite{indulal2010distance}} was introduced by Indulal et al., and it is defined as $E_D(G_n)= \sum_{i=1}^{n}|\mu_i|.$ Two graphs of the same order having equal $D$-energy are called \textit{$D$-equienergetic graphs}. In the literature \textnormal{\cite{ilic2010distance,ramane2008distance,central}}, we can see the constructions of different $D$-equienergetic graphs. In Section $2$, we provide a new class of $D$-equienergetic graphs of diameter $3$ using the above operation.\\
 \par The study of graphs through their topological descriptors is beneficial for deriving their underlying topologies. This process has many applications in modeling physico-chemical properties of molecules by QSAR/QSPR analysis (quantitative structure-
activity relationships/quantitative structure-property relationships).  In the studies, QSARs and QSPRs, the topological descriptors of graphs are used to approximate the biological activities and properties of chemical compounds. Various topological indices of chemical structures have emerged as a result of successful applications of QSAR/QSPR studies. This article gives some distance-based topological indices of the new graph operation central vertex-edge join of three graphs.\\ 
 The \textit{eccentric connectivity index} of $G_n$ is introduced in \textnormal{\cite{sharma}} as
\begin{equation}\tag{1}
    \xi^c(G_n)= \sum_{v_i\in V(G_n)}deg_{G_n}(v_i)ec_{G_n}(v_i).
\end{equation}
The \textit{connective eccentricity index} \textnormal{\cite{gupta}} of $G_n$ is
\begin{equation}\tag{2}
    \xi^{ce}(G_n)= \sum_{v_i\in V(G_n)} \frac{deg_{G_n}(v_i)}{ec_{G_n}(v_i)}.
\end{equation}

By using the eccentric connectivity index, the \textit{total-eccentricity index} \textnormal{\cite{ashrafi}} is given by
\begin{equation}\tag{3}
    \tau(G_n)= \sum_{v_i\in V(G_n)}ec_{G_n}(v_i).
\end{equation}

The mean value of eccentricity of vertices in $V(G_n)$ is called the \textit{average eccentricity} $aveg(G_n)$ \textnormal{\cite{skoro}},
\begin{equation}\tag{4}
    aveg(G_n)= \frac{1}{n}\sum_{v_i\in V(G_n)}ec_{G_n}(v_i)= \frac{\tau(G_n)}{n}.
\end{equation}

The \textit{Zagreb indices} \textnormal{\cite{ghor}} of $G_n$ in terms of eccentricities is
\begin{equation}\tag{5}
    \begin{aligned}
    M_1(G_n)&= \sum_{v_i\in V(G_n)} ec_{G_n}^{2}(v_i),\\
    M_2(G_n)&= \sum_{v_iv_j\in E(G_n)} ec_{G_n}(v_i)ec_{G_n}(v_j).
    \end{aligned}
\end{equation}
The \textit{geometric-arithmetic index} \textnormal{\cite{khaki}} in terms of eccentricity is given by
\begin{equation}\tag{6}
    GA_{4}(G_n)= \sum_{v_iv_j\in E(G_n)}\frac{2\sqrt{ec_{G_n}(v_i)ec_{G_n}(v_j)}}{ec_{G_n}(v_i)+ec_{G_n}(v_j)}.
\end{equation}

The \textit{atom-bond connectivity index} \textnormal{\cite{fara}} in terms of eccentricity is given by
\begin{equation}\tag{7}
    ABC_{5}(G_n)= \sum_{v_iv_j\in E(G_n)}\sqrt{\frac{ec_{G_n}(v_i)+ec_{G_n}(v_j)-2}{ec_{G_n}(v_i)ec_{G_n}(v_j)}}.
\end{equation}

The \textit{Wiener index} $W(G)$ \textnormal{\cite{wiener1947structural}} of $G_n$ is defined as 
\begin{equation}\tag{8}
    W(G_n)= \frac{1}{2}\sum_{v_i, v_j\in V(G_n)} d_{G}(v_i, v_j)= \frac{1}{2}\sum_{v_i\in V(G_n)}Tr(v_i).
\end{equation}

Section $3$ gives the eccentric connectivity index, connective eccentricity index, total-eccentricity index, average eccentricity index, Zagreb indices, geometric-arithmetic index, atom-bond connectivity index and Wiener index of central vertex-edge join of three graphs.
\section{Distance spectrum of CVE-join of three regular graphs}
In this section, we discuss the $D$-spectrum of $G_{n_1}^C \triangleright (G_{n_2}^V\cup G_{n_3}^E)$, where $G_{n_1}$ is a triangle-free regular graph and $G_{n_2}$, $G_{n_3}$ are two regular graphs.

\begin{lemma}\textnormal{\cite{cvet}}
 Let $G_n$ be a $k$-regular graph of order $n$ and size $m$ with $Spec(G_n)= \{\theta_1, \theta_2,\ldots, \theta_n\}$. Then 
 
 \[Spec (L(G_n))=\begin{pmatrix}
2k-2 & \theta_2+k-2 & \dots & \theta_n+k-2 & -2\\

1&1&\dots&1&m-n\\
\end{pmatrix}.\]
Moreover, $-2$ is an eigenvalue of $L(G_n)$ with eigenvector $V$ if and only if $Q(G_n)V=0.$
\end{lemma}
\begin{theorem}
 For $i=1,2,3$, let $G_{n_i}$ be a $k_i$-regular graph of order $n_i$ and size $m_i$, where $G_{n_1}$ is triangle-free. Then the $D$-eigenvalues of $G_{n_1}^C \triangleright (G_{n_2}^V\cup G_{n_3}^E)$ are the following\\
 
 \begin{itemize}
     \item[(i)]$\frac{-3+\theta_{1j}\pm \sqrt{(\theta_{1j}+1)^2+4(\theta_{1j}+k_1)}}{2}$, $j= 2, 3,\ldots, n_1. $
     \item[(ii)] $-2$ with multiplicity $m_1-n_1.$
     \item[(iii)]$-\theta_{ij}-2, i=2,3$ and $j= 2, 3,\ldots, n_i.$
     \item[(iv)]eigenvalues of the matrix \[\begin{bmatrix} n_1-1+k_1& 2m_1-k_1& n_2& 2n_3\\2n_1-2& 2m_1-2& 2n_2& n_3\\n_1& 2m_1& 2n_2-k_2-2& 3n_3\\2n_1& m_1& 3n_2& 2n_3-k_3-2\end{bmatrix}.\]
 \end{itemize}
 
\end{theorem}
\begin{proof}
  By a proper ordering of vertices in $G_{n_1}^C \triangleright (G_{n_2}^V\cup G_{n_3}^E)$, the distance matrix $D(G_{n_1}^C \triangleright (G_{n_2}^V\cup G_{n_3}^E))$ can be written as\\
  
 $\begin{aligned}
 &D(G_{n_1}^C \triangleright (G_{n_2}^V\cup G_{n_3}^E))\\
 &= \scriptsize{\begin{bmatrix} J_{n_1}-I_{n_1}+A(G_{n_1}) & 2J_{n_1\times m_1}-Q(G_{n_1}) & J_{n_1\times n_2}& 2J_{n_1\times n_3} \\ 2J_{m_1\times n_1}-Q(G_{n_1})^T & 2(J_{m_1}-I_{m_1}) & 2J_{m_1\times n_2}& J_{m_1\times n_3}\\ J_{n_2\times n_1} & 2J_{n_2\times m_1} & 2(J_{n_2}-I_{n_2})-A(G_{n_2})& 3J_{n_2\times n_3}\\2J_{n_3\times n_1}& J_{n_3\times m_1}& 3J_{n_3\times n_2}& 2(J_{n_3}-I_{n_3})-A(G_{n_3}) \end{bmatrix}}.
 \end{aligned}$\\
 \\
  Let $k_i=\theta_{i1}\geq \theta_{i2}\geq\dots\geq \theta_{in_i}$ be the $A$-eigenvalues of $G_{n_i}$, where $i = 1, 2, 3.$ Let $U$ be a vector such that $A(G_{n_1})U= \theta_{1j}U$ (for $j= 2, 3,\ldots, n_1$), then by Perron-Frobenius theory $J_{1\times n_{1}}U= 0$.\\
 Now let, $\phi_1= \begin{bmatrix} tU\\Q(G_{n_1})^TU\\0\\0\end{bmatrix}$ be an eigenvector of $D(G_{n_1}^C \triangleright (G_{n_2}^V\cup G_{n_3}^E))$ for the eigenvalue $\mu$, where $t$ is a real number. Then $D(G_{n_1}^C \triangleright (G_{n_2}^V\cup G_{n_3}^E))\phi_1= \mu \phi_1$.\\
 That is,
 \\
\begin{equation*}
      \scriptsize{\begin{aligned}
     &\begin{bmatrix}J_{n_1}-I_{n_1}+A(G_{n_1}) & 2J_{n_1\times m_1}-Q(G_{n_1}) & J_{n_1\times n_2}& 2J_{n_1\times n_3} \\ 2J_{m_1\times n_1}-Q(G_{n_1})^T & 2(J_{m_1}-I_{m_1}) & 2J_{m_1\times n_2}& J_{m_1\times n_3}\\ J_{n_2\times n_1} & 2J_{n_2\times m_1} & 2(J_{n_2}-I_{n_2})-A(G_{n_2})& 3J_{n_2\times n_3}\\2J_{n_3\times n_1}& J_{n_3\times m_1}& 3J_{n_3\times n_2}& 2(J_{n_3}-I_{n_3})-A(G_{n_3}) \end{bmatrix}\begin{bmatrix} tU\\Q(G_{n_1})^TU\\0\\0\end{bmatrix}\\
     &= \mu\begin{bmatrix} tU\\Q(G_{n_1})^TU\\0\\0\end{bmatrix}\\
     \\
     &\begin{bmatrix}(-t+\theta_{1j}t-\theta_{1j}-k_1)U\\(-t-2)Q(G_{n_1})^TU\\0\\0\end{bmatrix}= \begin{bmatrix} \mu tU\\\mu Q(G_{n_1})^TU\\0\\0\end{bmatrix}.
     \end{aligned}}
 \end{equation*}
     
Therefore,
 \begin{equation*}
\begin{aligned}
\quad -t+\theta_{1j}t-\theta_{1j}-k_1&= \mu t\\
 \quad-t-2&= \mu
	\end{aligned}
\end{equation*}
from this, we get $$t^2+(\theta_{1j}+1)t-(\theta_{1j}+k_1)= 0$$
 so that $t$ has two values
\begin{equation*}
\begin{aligned}
&t_1 = \frac{-(\theta_{1j}+1)+ \sqrt{(\theta_{1j}+1)^2+4(\theta_{1j}+k_1)}}{2},\\
& t_2 = \frac{-(\theta_{1j}+1)- \sqrt{(\theta_{1j}+1)^2+4(\theta_{1j}+k_1)}}{2}.
 \end{aligned}
\end{equation*}
That is,  $-t_1-2$ and $-t_2-2$ are eigenvalues of $D(G_{n_1}^C \triangleright (G_{n_2}^V\cup G_{n_3}^E))$ corresponding to the eigenvalues $\theta_{1j}\neq k_1$ of $A(G_{n_1})$. Thus, we get $2(n_1-1)$ eigenvalues of $D(G_{n_1}^C \triangleright (G_{n_2}^V\cup G_{n_3}^E))$.\\
  By Lemma $2.1$, we get $Q(G_{n_1})V=0$, where $V$ is an eigenvector of $L(G_{n_1})$ corresponding to the eigenvalue $-2$ (repeated $m_1-n_1$ times).\\ If $\phi_2= \begin{bmatrix}0\\V\\0\\0\end{bmatrix}$, then we have $$D(G_{n_1}^C \triangleright (G_{n_2}^V\cup G_{n_3}^E))\phi_2= -2\phi_2.$$ Therefore, $-2$ is an eigenvalue of $D(G_{n_1}^C \triangleright (G_{n_2}^V\cup G_{n_3}^E))$ repeated $m_1-n_1$ times.\\
 Consider the eigenvalue $\theta_{2j}$ (for $j= 2,3,\ldots,n_2)$ of $G_{n_2}$ with an eigenvector $W$ such that $J_{1\times n_2}W= 0$.\\ Put $\phi_3= \begin{bmatrix}0\\0\\W\\0 \end{bmatrix}$, then we have
  $$D(G_{n_1}^C \triangleright (G_{n_2}^V\cup G_{n_3}^E))\phi_3= -(\theta_{2j} +2)\phi_3.$$This shows that $-(\theta_{2j} +2)$ are eigenvalues of $D(G_{n_1}^C \triangleright (G_{n_2}^V\cup G_{n_3}^E))$ corresponding to the eigenvalues $\theta_{2j}\neq k_2$ of $A(G_{n_2})$.\\
  Similarly, for each eigenvalue $\theta_{3j}$ (for $j= 2,3,\ldots,n_3$) of $G_{n_3}$, we get an eigenvalue $-(\theta_{3j} +2)$ of $D(G_{n_1}^C \triangleright (G_{n_2}^V\cup G_{n_3}^E))$ with eigenvector $\begin{bmatrix}0\\0\\0\\X\end{bmatrix}.$ 
  Thus, we get $2n_1-2+m_1-n_1+n_2-1+n_3-1= n_1+m_1+n_2+n_3-4$ eigenvalues of $D(G_{n_1}^C \triangleright (G_{n_2}^V\cup G_{n_3}^E))$.
  All the other eigenvalues of $D(G_{n_1}^C \triangleright (G_{n_2}^V\cup G_{n_3}^E))$ are the eigenvalues of the equitable quotient matrix
\[M= \begin{bmatrix} n_1-1+k_1& 2m_1-k_1& n_2& 2n_3\\2n_1-2& 2m_1-2& 2n_2& n_3\\n_1& 2m_1& 2n_2-k_2-2& 3n_3\\2n_1& m_1& 3n_2& 2n_3-k_3-2\end{bmatrix}.\]
\end{proof}

Using the above theorem we present some new families of $D$-equienergetic graphs.\\
\\
 For a fixed $a\in \mathbb N $, let $\mathcal {P}_a$ be the family of all integer partitions of $'a'$, denoted by $p_1, p_2,\ldots, p_s$, each of size at least $3$. For $P= \{p_1, p_2,\ldots, p_s \}\in \mathcal {P}_a$, let $C_P$ be the union of cycles with vertices $p_1, p_2,\ldots, p_s$.
\begin{theorem}
 Let $H_{n_1}$ be a $k_1$-regular triangle free graph of order $n_1$ and size $m_1$ and $H_{n_2}$ be a $k_2$- regular graph of order $n_2$ and size $m_2$ with the least eigenvalue being at least $-2$. Then for every $a\in \mathbb N$,
 $H_{n_1}^C \triangleright (C_P^V\cup H_{n_2}^E)$ forms a family of $D$-equienergetic graphs.
 
\end{theorem}
\begin{proof}
Let $k_1=\theta_1\geq\theta_2\geq\cdots\geq \theta_{n_1}$, $2=\lambda_1\geq \lambda_2\geq\cdots\geq\lambda_a$ and $k_2=\gamma_1\geq\gamma_2\geq\cdots\geq \gamma_{n_2}$ be the $A$-eigenvalues of $H_{n_1}$, $C_P$ and $H_{n_2}$ respectively. 
From Theorem $2.2$, the $D$-eigenvalues of $ H_{n_1}^C \triangleright (C_P^V\cup H_{n_2}^E)$ are $\frac{-3+\theta_j\pm \sqrt{(\theta_j+1)^2+4(\theta_j+k_1)}}{2}$, $j= 2, 3,\ldots, n_1 $; $-2$ with multiplicity $m_1-n_1$; $-\lambda_{j^{'}}-2$, $j^{'}=2,3,\ldots,a$; $-\gamma_{j^{''}}-2, j^{''}=2,3,\ldots,n_2$ and along with the eigenvalues of the equitable quotient matrix $M$ of $ H_{n_1}^C \triangleright (C_P^V\cup H_{n_2}^E)$, \[\begin{bmatrix} n_1+k_1-1& 2m_1-k_1& a& 2n_2\\ 2n_1-2& 2m_1-2& 2a& n_2\\ n_1& 2m_1& 2a-4& 3n_2\\ 2n_1& m_1& 3a& 2n_2-k_2-2\end{bmatrix}\]  for every partition $P$ of $'a'$.\\
  Since $\lambda_{j^{'}}\geq -2$, for $j^{'}=2,3,\ldots,a$, we have $-(\lambda_{j^{'}}+2)\leq 0$.
 Similarly, we get $-(\gamma_{j^{''}}+2)\leq 0, j^{''}=2,3,\ldots,n_2.$\\
 Therefore,\begin{equation*}
     \begin{aligned}
     \sum_{j^{'}=2}^{a}|{-\lambda_{j^{'}}-2}|+\sum_{j^{''}=2}^{n_2}|{-\gamma_{j^{''}}-2}|&= \sum_{j^{'}=2}^a(\lambda_{j^{'}}+2)+\sum_{j^{''}=2}^{n_2}(\gamma_{j^{''}}+2)\\
     &= -2+2(a-1)-k_2+2(n_2-1)\\
     &= 2a-k_2+2n_2-6.
     \end{aligned}
 \end{equation*}
 Hence for every partition of $'a'$, the energy remains the same.

\end{proof}

\section{Distance-based topological indices of CVE-join of three graphs}
For three graphs $G_{n_1},G_{n_2}, \text{and}\; G_{n_3}$, their CVE-join $\mathcal{G}= G_{n_1}^{C}\vee(G_{n_2}^V\cup G_{n_2}^E)$ has the following properties,\\
\begin{equation}\tag{9}
    deg_{\mathcal{G}}(v)= \begin{cases}n_1+n_2-1, &\text{if $v\in V(G_{n_1})$}\\n_3+2, &\text{if $v\in I(G_{n_1})$}\\deg_{G_{n_2}}(v)+n_1,&\text{if $v\in V(G_{n_2})$}\\deg_{G_{n_3}}(v)+m_1, &\text{if $v\in V(G_{n_3})$} \end{cases}.
\end{equation}

Now we give the eccentricity of vertices in $\mathcal{G}$ by two different cases;\\
\noindent
Case 1. If $G_{n_1}$ is a triangle free graph,\\
\begin{equation}\tag{10}
    ec_{\mathcal{G}}(v)=  \begin{cases}2, &\text{if $v\in V(G_{n_1})$or $v\in I(G_{n_1})$}\\3, &\text{if $v\in V(G_{n_2})$ or $v\in V(G_{n_3})$} \end{cases}.
\end{equation}
   
\noindent
Case 2. If $G_{n_1}$ is a graph which is not triangle free, then 
\begin{equation}\tag{11}
    ec_{\mathcal{G}}(v)=  3,\; \text{for every vertex}\; v\in V(\mathcal{G})
\end{equation}

This section gives some graph invariants of $\mathcal{G}= G_{n_1}^{C}\vee(G_{n_2}^V\cup G_{n_2}^E)$ like Wiener index, Zagreb indices, etc.
\begin{theorem}
 For $i= 1,2,3$, let $G_{n_i}$ be a $k_i$- regular graph of order $n_i$ and size $m_i$, where $G_{n_1}$ is triangle-free. Then the Wiener index of $\mathcal{G}= G_{n_1}^{C}\vee(G_{n_2}^V\cup G_{n_2}^E)$ is given by,\\
 $W(\mathcal{G})= \frac{1}{2}(n_1^2-n_1+2(n_2^2+n_3^2+n_1n_2-n_2-n_3+m_1n_3+m_1^2)+4(n_1n_3+m_1n_1+m_1n_2-m_1)-(n_2k_2+n_3k_3)+6n_2n_3).$
 
\end{theorem}
\begin{proof}
Let the vertex set of $\mathcal{G}= G_{n_1}^{C}\vee(G_{n_2}^V\cup G_{n_2}^E)$ be\\
$V(\mathcal{G})= \{v_1, v_2, \ldots, v_{n_1}, v_{e_1}, v_{e_2}, \ldots, v_{e_{m_1}}, u_1, u_2, \ldots, u_{n_2},  w_1, w_2, \ldots, w_{n_3} \}$, where $v_1, v_2, \ldots, v_{n_1}$ are the
vertices in $G_{n_1}$, $v_{e_1}, v_{e_2}, \ldots, v_{e_{m_1}}$ are the vertices corresponding to the edges of $G_{n_1}$ in
$C(G_{n_1})$, $u_1, u_2, \ldots, u_{n_2}$ are the vertices in $G_{n_2}$, and $w_1, w_2, \ldots, w_{n_3}$ are the vertices in $G_{n_3}$. Then
\begin{equation*}
    \begin{aligned}
    Tr(v_i)&= n_1+n_2+2n_3+2m_1-1,\; i= 1, \ldots, n_1.\\
    Tr(v_{e_j})&= 2n_1+2n_2+n_3+2m_1-4,\; j= 1, \ldots, m_1.\\
    Tr(u_k)&= n_1+2n_2+3n_3+2m_1-k_2-2,\; k= 1, \ldots, n_2.\\
    Tr(w_l)&= 2n_1+3n_2+2n_3+m_1-k_3-2,\; l= 1, \ldots, n_3.
    \end{aligned}
\end{equation*}
From (8),\\
\begin{equation*}
    \begin{aligned}
    W(\mathcal{G})
    &= \frac{1}{2}\sum_{v_i\in V(\mathcal{G})}Tr(v_i)\\
    &= \frac{1}{2}(n_1^2-n_1+2(n_2^2+n_3^2+n_1n_2-n_2-n_3+m_1n_3+m_1^2)\\
    &\; \; \;+4(n_1n_3+m_1n_1+m_1n_2-m_1)-(n_2k_2+n_3k_3)+6n_2n_3).
    \end{aligned}
\end{equation*}
\end{proof}
\begin{theorem}
 For $i= 1,2,3$, let $G_{n_i}$ be a graph of order $n_i$ and size $m_i$, then
 \begin{itemize}
     \item [(i)] $\xi^c(\mathcal{G})=\begin{cases} 2n_1^2+5n_1n_2+5m_1n_3-2n_1+4m_1+6m_2+6m_3, & \text{if $G_{n_1}$ is triangle-free}\\3n_1^2+6n_1n_2+6m_1n_3-3n_1+6(m_1+m_2+m_3), &\text{otherwise}\end{cases}.$
     \item[(ii)]$\xi^{ce}(\mathcal{G})= \begin{cases} \frac{1}{6}\left(3n_1^2+5(n_1n_2+m_1n_3)+6m_1+4(m_2+m_3)-3n_1\right), & \text{if $G_{n_1}$ is triangle-free}\\\frac{1}{3}\left(n_1^2+2(n_1n_2+m_1n_3)-n_1+2(m_1+m_2+m_3)\right), &\text{otherwise}\end{cases}.$
 \end{itemize}
  
\end{theorem}
\begin{proof}
\begin{itemize}
    \item [(i)] Let $\mathcal{G}= G_{n_1}^{C}\vee(G_{n_2}^V\cup G_{n_2}^E)$, where $G_{n_1}$ is triangle-free, then from (1) by using (9) and (10), we have\\
\begin{equation*}
    \begin{aligned}
    \xi^c(\mathcal{G})&= \sum_{v\in V(G_{n_1})}2(n_1+n_2-1)+ \sum_{v\in I(G_{n_1})}2(n_3+2)\\
    &\; \; \;+ \sum_{v\in V(G_{n_2})}3(deg_{G_{n_2}}(v)+n_1)+ \sum_{v\in V(G_{n_3})} 3(deg_{G_{n_3}}(v)+m_1)\\
    &= 2n_1^2+5n_1n_2+5m_1n_3-2n_1+4m_1+6m_2+6m_3.
    \end{aligned}
\end{equation*}
Similarly in the case when $G_{n_1}$ is not a triangle-free graph by using (9) and (10) we get,
$$\xi^c(\mathcal{G})= 3n_1^2+6n_1n_2+6m_1n_3-3n_1+6(m_1+m_2+m_3).$$
\item [(ii)] Let $G_{n_1}$ be a triangle-free graph. Then from (2), by using (9) and (10), we get\\
\begin{equation*}
    \begin{aligned}
    \xi^{ce}(\mathcal{G})&= \sum_{v\in V(G_{n_1})}\frac{n_1+n_2-1}{2}+\sum_{v\in I(G_{n_1})}\frac{n_3+2}{2}+\sum_{v\in V(G_{n_2})}\frac{deg_{G_{n_2}} +n_1}{3}\\
    &\; \; \;+\sum_{v\in V(G_{n_3})}\frac{deg_{G_{n_3}}+m_1}{3}\\
    &= \frac{1}{6}\left(3n_1^2+5(n_1n_2+m_1n_3)+6m_1+4(m_2+m_3)-3n_1\right).
    \end{aligned}
\end{equation*}
Similarly in the case when $G_{n_1}$ is not a triangle-free graph by using (9) and (10) we get,
$$\xi^{ce}(\mathcal{G})= \frac{1}{3}\left(n_1^2+2(n_1n_2+m_1n_3)-n_1+2(m_1+m_2+m_3)\right).$$
\end{itemize}

\end{proof}
\begin{theorem}
  For $i= 1,2,3$, let $G_{n_i}$ be a graph of order $n_i$ and size $m_i$, then
 \begin{itemize}
     \item [(i)] $\tau(\mathcal{G})=\begin{cases} 2n_1+2m_1+3n_2+3n_3, & \text{if $G_{n_1}$ is triangle-free}\\3(n_1+n_2+n_3+m_1), &\text{otherwise}\end{cases}.$
     \item[(ii)]$aveg(\mathcal{G})= \begin{cases} \frac{2n_1+2m_1+3n_2+3n_3}{n_1+n_2+n_3+m_1}, & \text{if $G_{n_1}$ is triangle-free}\\3, &\text{otherwise}\end{cases}.$
 \end{itemize}
\end{theorem}
\begin{proof}
\begin{itemize}
    \item [(i)] Consider $\mathcal{G}= G_{n_1}^{C}\vee(G_{n_2}^V\cup G_{n_2}^E)$, when $G_{n_1}$ is triangle-free, then from (3) by using (10), we get\\
    \begin{equation*}
        \begin{aligned}
        \tau(\mathcal{G})&= 2n_1+2m_1+3n_2+3n_3.
        \end{aligned}
    \end{equation*}
    Similarly when $G_{n_1}$ is not a triangle-free graph we get\\
    \begin{equation*}
        \begin{aligned}
        \tau(\mathcal{G})&= 3(n_1+n_2+n_3+m_1).
        \end{aligned}
    \end{equation*}
    \item[(ii)] From (4), we get\\
    \begin{equation*}
        \begin{aligned}
        aveg(\mathcal{G})&= \frac{\tau(\mathcal{G})}{n_1+n_2+n_3+m_1}\\
        &= \begin{cases} \frac{2n_1+2m_1+3n_2+3n_3}{n_1+n_2+n_3+m_1}, & \text{if $G_{n_1}$ is triangle-free}\\3, &\text{otherwise}\end{cases}.
        \end{aligned}
    \end{equation*}
\end{itemize}
\end{proof}
\begin{theorem}
  For $i= 1,2,3$, let $G_{n_i}$ be a graph of order $n_i$ and size $m_i$, then
\begin{itemize}
     \item [(i)] $M_1(\mathcal{G})=\begin{cases} 4n_1+4m_1+9n_2+9n_3, & \text{if $G_{n_1}$ is triangle-free}\\9(n_1+n_2+n_3+m_1), &\text{otherwise}\end{cases}.$
     \item[(ii)]$M_2(\mathcal{G})= \begin{cases}4m_1+9(m_2+m_3)+2(n_1^2-n_1)+6(n_1n_2+m_1n_3) , & \text{if $G_{n_1}$ is triangle-free}\\9\left(m_1+m_2+m_3+n_1n_2+m_1n_3+\frac{n_1^2-n_1}{2}\right), &\text{otherwise}\end{cases}.$
 \end{itemize}
\end{theorem}
\begin{proof}
\begin{itemize}
    \item [(i)] Let $\mathcal{G}= G_{n_1}^{C}\vee(G_{n_2}^V\cup G_{n_2}^E)$, where $G_{n_1}$ is triangle-free, then from (5) by using (10) we get\\
    \begin{equation*}
        \begin{aligned}
        M_1(\mathcal{G})&= \sum_{v\in V(G_{n_1})}2^2+\sum_{v\in I(G_{n_1})}2^2+\sum_{v\in V(G_{n_2})}3^2+\sum_{v\in V(G_{n_3})}3^2\\
        &= 4n_1+4m_1+9n_2+9n_3.
        \end{aligned}
    \end{equation*}
    Similarly when $G_{n_1}$ is not a triangle-free graph we get,\\
    $$M_1(\mathcal{G})= 9(n_1+n_2+n_3+m_1).$$
    \item[(ii)] Let $G_{n_1}$ be a triangle-free graph. Then from (5) by using (10) we get\\
    \begin{equation*}
        \begin{aligned}
        M_2(\mathcal{G})&= \sum_{uv\in E(G_{n_2})}3\times 3+\sum_{uv\in E(G_{n_3})}3\times 3+\sum_{{\substack{u\in V(G_{n_1}), \\ v\in V(G_{n_2})}}}2\times 3\\
        &\; \; +\sum_{{\substack{u\in I(G_{n_1}), \\ v\in V(G_{n_3})}}}2\times 3+\sum_{{\substack{uv\in E(\mathcal{G}),\\u\in V(G_{n_1}), v\in I(G_{n_1})}}}2\times 2
        +\sum_{{\substack{uv\in E(\mathcal{G}),\\u,v\in V(G_{n_1})}}}2\times 2\\
        &= 4m_1+9m_2+9m_3+2n_{1}^{2}-2n_1+6n_1n_2+6m_1n_3.
        \end{aligned}
    \end{equation*}
     Similarly when $G_{n_1}$ is not a triangle-free graph we get,\\
    $$M_2(\mathcal{G})= 9\left(m_1+m_2+m_3+n_1n_2+m_1n_3+\frac{n_1^2-n_1}{2}\right).$$
\end{itemize}

\end{proof}
\begin{theorem}
  For $i= 1,2,3$, let $G_{n_i}$ be a graph of order $n_i$ and size $m_i$, then
 \begin{itemize}
     \item [(i)] $GA_{4}(\mathcal{G})=\begin{cases} m_1+m_2+m_3+\frac{2\sqrt{6}}{5}(n_1n_2+m_1n_3)+\frac{n_1^2-n_1}{2}, & \text{if $G_{n_1}$ is triangle-free}\\m_1+m_2+m_3+n_1n_2+m_1n_3+\frac{n_1^2-n_1}{2}, &\text{otherwise}\end{cases}.$
     \item[(ii)]$ABC_{5}(\mathcal{G})= \begin{cases}\frac{2}{3}(m_2+m_3)+\frac{1}{\sqrt{2}}\left(n_1n_2+m_1n_3+m_1+\frac{n_1(n_1-1)}{2}\right) , & \text{if $G_{n_1}$ is triangle-free}\\\frac{2}{3}(m_1+m_2+m_3+n_1n_2+m_1n_3)+\frac{n_1(n_1-1)}{3}, &\text{otherwise}\end{cases}.$
 \end{itemize}
\end{theorem}
\begin{proof}
\begin{itemize}
    \item [(i)]Let $\mathcal{G}= G_{n_1}^{C}\vee(G_{n_2}^V\cup G_{n_2}^E)$, where $G_{n_1}$ is triangle-free, then from (6) by using (10) we have\\
    \begin{equation*}
        \begin{aligned}
        GA_{4}(\mathcal{G})&= \sum_{uv\in E(G_{n_2})}2\frac{\sqrt{3\times 3}}{3+3}+\sum_{uv\in E(G_{n_3})}2\frac{\sqrt{3\times 3}}{3+3}+\sum_{{\substack{u\in V(G_{n_1}),\\ v\in V(G_{n_2})}}}2\frac{\sqrt{2\times 3}}{2+3}\\& +\sum_{{\substack{u\in I(G_{n_1}),\\ v\in V(G_{n_3})}}}2\frac{\sqrt{2\times 3}}{2+3}+\sum_{{\substack{u\in V(G_{n_1}),\\ v\in I(G_{n_1}),\\ uv\in E(\mathcal{G})}}}2\frac{\sqrt{2\times 2}}{2+2}+\sum_{{\substack{uv\in E(\mathcal{G}),\\  u,v\in V(G_{n_1})}}}2\frac{\sqrt{2\times 2}}{2+2}\\
        &= m_1+m_2+m_3+\frac{2\sqrt{6}}{5}(n_1n_2+m_1n_3)+\frac{n_1^2-n_1}{2}.
        \end{aligned}
    \end{equation*}
    Similarly when $G_{n_1}$ is not a triangle-free graph we get,\\
    $$GA_{4}(\mathcal{G})= m_1+m_2+m_3+n_1n_2+m_1n_3+\frac{n_1^2-n_1}{2}.$$
    \item[(ii)] Let $G_{n_1}$ be a triangle-free graph. Then from (7) by using (10) we get\\
    \begin{equation*}
       \scriptsize{ \begin{aligned}
        ABC_{5}(\mathcal{G})&= \sum_{uv\in E(G_{n_2})}\sqrt{\frac{3+3-2}{3\times 3}}+\sum_{uv\in E(G_{n_3})}\sqrt{\frac{3+3-2}{3\times 3}}+\sum_{{\substack{u\in V(G_{n_1}),\\  v\in V(G_{n_2})}}}\sqrt{\frac{2+3-2}{2\times 3}}\\& +\sum_{{\substack{u\in I(G_{n_1}),\\   v\in V(G_{n_3})}}}\sqrt{\frac{2+3-2}{2\times 3}}+\sum_{{\substack{uv\in E(\mathcal{G}),\\u\in V(G_{n_1}),\\ v\in I(G_{n_1})}}}\sqrt{\frac{2+2-2}{2\times 2}}+\sum_{{\substack{uv\in E(\mathcal{G}) \\ u,v\in V(G_{n_1})}}}\sqrt{\frac{2+2-2}{2\times 2}}\\
        &= \frac{2}{3}(m_2+m_3)+\frac{1}{\sqrt{2}}\left(n_1n_2+m_1n_3+m_1+\frac{n_1(n_1-1)}{2}\right).
        \end{aligned}}
    \end{equation*}
     Similarly when $G_{n_1}$ is not a triangle-free graph we get,\\
    $$ABC_{5}(\mathcal{G})= \frac{2}{3}(m_1+m_2+m_3+n_1n_2+m_1n_3)+\frac{n_1(n_1-1)}{3}.$$
\end{itemize}
\end{proof}
\section{Conclusion}
Topological indices play an important role in mathematical chemistry. In QSAR/QSPR study, topological indices can be used to estimate the bioactivity of chemical compounds. In modeling the chemical and other properties of molecules, we can use topological indices as a tool. This article gives the results related to topological indices like the eccentric connectivity index, connective eccentricity index, total-
eccentricity index, average eccentricity index, Zagreb indices, geometric-arithmetic index,
atom-bond connectivity index, and Wiener index of the newly defined graph operation called central vertex-edge join (CVE-join). Also, the $D$-spectrum of CVE-join of three regular graphs $G_{n_1}$, $G_{n_2}$ and $G_{n_3}$ when $G_{n_1}$ is a triangle-free are obtained. These results enable us to construct new $D$-equienergetic graph families. 

\section*{Acknowledgements} 
The authors would like to thank the DST, Government of India, for providing support to carry out this work under the scheme 'FIST'(No.SR\\/FST/MS-I/2019/40).

%\bibliographystyle{plain}
%\bibliography{seidel}

\end{document}